\documentclass[12pt]{article}

\usepackage{a4wide}
\usepackage{amsmath}
\usepackage{amssymb}
\usepackage{amsthm}
\usepackage{latexsym}
\usepackage{graphicx}
\usepackage[english]{babel}
\usepackage{makeidx}
\usepackage{arcs}
\usepackage{pst-node}

\newtheorem{dfn} [subsection]{Definition}
\newtheorem{obs} [subsection]{Remark}

\newtheorem{exm} [subsection]{Example}

\newtheorem{prop}[subsection]{Proposition}

\newtheorem{teor}[subsection]{Theorem}

\def\Sind{\operatorname{SInd}}
\def\Gal{\operatorname{Gal}}
\def\Irr{\operatorname{Irr}}
\def\SCl{\operatorname{SCl}}
\def\Sup{\operatorname{Sup}}
\def\Reg{\operatorname{Reg}}
\def\ord{\operatorname{ord}}
\numberwithin{equation}{section}

\title{On arithmetic Heilbronn supercharacters}

\author{Mircea Cimpoea\c s$^1$}
\date{}

\begin{document}

\maketitle
\footnotetext[1]{ \emph{Mircea Cimpoea\c s}, National University of Science and Technology Politehnica Bucharest, Faculty of
Applied Sciences, 
Bucharest, 060042, Romania and Simion Stoilow Institute of Mathematics, Research unit 5, P.O.Box 1-764,
Bucharest 014700, Romania, E-mail: mircea.cimpoeas@upb.ro,\;mircea.cimpoeas@imar.ro }

\begin{abstract}
In this note, we introduce arithmetic Heilbronn supercharacters that generalize the notions of
arithmetic Heilbronn characters and Heilbronn supercharacters and discuss several properties of them.
 
\textbf{Keywords:} Heilbronn character; supercharacter theory; L-functions.

\textbf{MSC2020: 11R42; 20C15} 
\end{abstract}

\setcounter{section}{0}
\section*{Introduction}

Let $\mathbb Q \subset K$ be a number field. In order to study the zeros of the Dedekind zeta function $\zeta_K(s)$, Heilbronn
\cite{heilb} introduced what are now called Heilbronn characters, which allowed him to give a simple proof of the famous
Aramata and Brauer Theorem \cite{arama,brauer}, that is $\zeta_K(s)/\zeta(s)$ is entire.

More generally, if $K\subset L$ is a number field externsion, the problem if $\zeta_L(s)/\zeta_K(s)$ is entire is open, and it
would be a consequence of the Artin conjecture for L-functions \cite{artin2}. This connection
suggests that the method of Heilbronn is useful in studying $L$-functions and, indeed, it was used by several authors; see
for instance \cite{murty}.

In \cite{heil}, P.-J. Wong introduced the so called arithmetic Heilbronn characters which generalize the classical
Heilbronn characters and, at the same time, catch almost all properties of them. The aim of this note is to further generalize
the arithmetic Heilbronn characters, in the framework of the supercharacter theory, introduced by Diaconis and Isaacs \cite{isaacs}.

In Definition \ref{defi} we introduce the notion of arithmetic Heilbronn supercharacters, which generalizes both the arithmetic
Heilbronn characters and the Heilbronn supercharaters; see \cite[Section 4.1]{pong}. Using the supercharacter theoretic
formalism, we prove several generalizations of some classical results: Artin-Takagi decomposition (Theorem \ref{t1}), 
Heilbronn-Stark Lemma (Theorem \ref{t2}) and Uchida-van der Waall Theorem (Theorem \ref{t3}).

\pagebreak

\section{Preliminaries}

\begin{dfn}
Let $G$ be a finite group. Let $\mathcal K$ be a partition of $G$ and let $\mathcal X$ be a partition of $\Irr(G)$. 
The ordered pair $C:=(\mathcal X,\mathcal K)$
is a \emph{supercharacter theory} if:
\begin{enumerate}
\item $\{1\}\in\mathcal K$,
\item $|\mathcal X|=|\mathcal K|$, and
\item for each $X\in\mathcal X$, the character $\sigma_X=\sum_{\psi\in X}\psi(1)\psi$ is constant on each $K\in\mathcal K$.
\end{enumerate}
The characters $\sigma_X$ are called \emph{supercharacters}, and the elements $K$ in $\mathcal K$ are called \emph{superclasses}.
We denote $\Sup(G)$ the set of supercharacter theories of $G$.
\end{dfn}

Diaconis and Isaacs showed their theory enjoys properties similar to the classical character theory.
For example, every superclass is a union of conjugacy classes in $G$; see \cite[Theorem 2.2]{isaacs}.
The irreducible characters and conjugacy classes of $G$ give a supercharacter 
theory of $G$, which will be referred to as the \emph{classical theory} of $G$.

Also, as noted in \cite{isaacs}, every group $G$ admits
a non-classical theory with only two supercharacters $1_G$ and $\Reg(G)-1_G$ and
superclasses $\{1\}$ and $G\setminus\{1\}$, where $1_G$ denotes the trivial character of $G$
and $$\Reg(G) = \sum_{\chi\in\Irr(G)}\chi(1)\chi$$ is the regular character of $G$. 
This theory will be called the \emph{maximal theory} of $G$.



Let $C_G\in\Sup(G)$ be a supercharacter theory of $G$, $C_G=(\mathcal X_G,\mathcal K_G)$. Let $g\in G$. 
We denote by $\SCl_G(g)$, the superclass of $G$ which contain $g$.

\begin{dfn}(\cite[Definition 2.7]{pong})
Let $G$ be a finite group and $H$ a subgroup of $G$. Let $C_G\in\Sup(G)$ be a supercharacter theory of $G$ and $C_H\in\Sup(H)$
a supercharacter theory of $H$. We say that $C_G$ and $C_H$ are compatible if for any $h\in H$, we have
$$\SCl_H(h)\subseteq\SCl_G(h).$$
Moreover, if $C_H$ and $C_G$ are compatible and $\Phi:H\to\mathbb C$ is a superclass function of $H$, i.e. a function
constant on superclasses of $H$, then the superinduction $\Phi^G:G\to\mathbb C$ is defined by
$$\Sind_H^G \Phi^G(g)=\frac{|G|}{|H|\cdot|\SCl_G(g)|}\sum_{x\in \SCl_G(g)}\Phi^0(x),$$
where $\Phi^0(x)$ denotes $\Phi(x)$ if $x\in H$ and zero otherwise.
\end{dfn}

\begin{obs}\rm
Let $H_1\subset H_2\subset G$ be a chain of subgroups of $G$ and let $C_{H_1},C_{H_2}$ and $C_G$ be supercharacter theories of $H_1$, $H_2$
and, respectively, $G$. If $C_{H_1}$ and $C_{H_2}$ are compatible and $C_{H_2}$ and $C_G$ are compatible then $C_{H_1}$ and $C_G$ are also
compatible.
\end{obs}

We recall the following result, see \cite[Proposition 2.14]{pong}:

\begin{prop} (Super Frobenius Reciprocity)

Let $G$ be a finite group and $H$ a subgroup of $G$. Let $C_G\in\Sup(G)$ and $C_H\in\Sup(H)$ such that $C_G$ and $C_H$ are compatible.
For all superclass functions $\Phi$ on $H$ and all superclass functions $\theta$ on $G$,
$$ \langle \Sind_H^G\Phi^G,\theta \rangle = \langle \Phi,\theta|_H \rangle,$$
where $\theta|_H$ is the restriction of $\theta$ from $G$ to $H$.
\end{prop}

As it was noted in \cite{pong}, the superinduction is unique, in the sense that it satisfies the Super Frobenius Reciprocity. More precisely, if
$\Phi \to \Phi^{(G)}$ is another arbitrary map sending superclass functions of $H$ to superclass functions of $G$ such that
$$ \langle \Phi^{(G)}, \theta \rangle  = \langle \Phi, \theta|_H \rangle, $$ 
for any super class function $\theta$ of $G$, it follows that $\Phi^{(G)}=\Sind_H^G\Phi^G$.

\section{Main results}

\begin{dfn}
Let $G$ be a finite group. Let $\mathcal C:=\{C_H\in \Sup(H)\;:\;H\leqslant G\}$ be a family of supercharacter theories on the all
subgroups of $G$. We say that $\mathcal C$ is compatible if, for any subgroups $H_1\subset H_2$ of $G$, $C_{H_1}$ and $C_{H_2}$
are compatible.
\end{dfn}

We introduce the following generalization of \cite[Definition 3.1]{heil}, in the framework of supercharacter theory:

\begin{dfn}\label{defi}
Let $G$ be a finite group and $\mathcal C$ be a compatible family of supercharacter theories on the subgroups of $G$.
Let $$I(G,\mathcal C)=\{(H,\sigma)\;:\;\sigma\text{ a supercharacter of }H\}.$$
Suppose that there is a set of integers $\{n(H, \sigma)\;:\; (H,\sigma)\in I(G,\mathcal C)\}$
satisfying the following three properties:
\begin{enumerate}
\item[ACH1:] $n(H,\sigma_1+\sigma_2)=n(H,\sigma_1)+n(H,\sigma_2)$ for any subgroup $H$ of $G$ and any
            supercharacters $\sigma_1$ and $\sigma_2$ of $H$.
\item[ACH2:] $n(G,\Sind_H^G\sigma^G)=n(H,\sigma)$ for any subgroup $H$ of $G$ and any supercharacter $\sigma$ of $H$.
\item[ACH3:] $n(H,\sigma)\geq 0$ for any supercharacter $\sigma$ of a subgroup $H$ of $G$ with linear constituents, that is $\sigma=\lambda_1+\cdots+\lambda_m$,
            where $\lambda_i$'s are linear characters of $H$.
\end{enumerate}
Then the arithmetic Heilbronn supercharacter of a subgroup $H$ of $G$ associated to $n(H, \sigma)$'s is
$$\Theta_H:=\sum_{X\in \mathcal X_H}\frac{n(H,\sigma_X)}{\sigma_X(1)}\sigma_X,$$
where $C_H=(\mathcal X_H,\mathcal K_H)$ and $\sigma_X=\sum_{\chi\in X}\chi(1)\chi$.
\end{dfn}

\begin{prop}
With the above notations, we have that
$$ \Theta_H=\sum_{X\in \mathcal X_H} \frac{n(G,\Sind_H^G\sigma_X^G)}{\sigma_X(1)}\sigma_X.$$
\end{prop}

\begin{proof}
It follows immediately from $ACH2$.
\end{proof}

\begin{exm}\rm
Let $K/\mathbb Q$ be a Galois extension with Galois group $G:=\Gal(K/\mathbb Q)$. Let $\mathcal C$ be a compatible family of supercharacter theories on the subgroups of $G$.
For any subgroup $H$ of $G$ and any supercharacter $\sigma$ of $H$, we define:
$$n(H,\sigma):=\ord_{s=s_0}L(s,\sigma,K/K^H),$$
where $L(s,\sigma,K/K^H)$ is the $L$-Artin function associated to the extension $K^H\subset K$ and $s_0\in\mathbb C\setminus\{1\}$ is a fixed point.

Then the integers $n(H,\sigma)$'s satisfy the conditions $ACH1$, $ACH2$ and $ACH3$.
\end{exm}

\begin{teor}(Artin-Takagi decomposition)\label{t1}
We have that:
$$n(G,\Reg(G))=\sum_{X\in\mathcal X_G}n(G,\sigma_X)=\sum_{X\in\mathcal X_G} \frac{n(G,\sigma_X)}{\sigma_X(1)} \langle \sigma_X,\sigma_X \rangle .$$
\end{teor}

\begin{proof}
Since $\Reg(G)=\sum_{\chi\in\Irr(G)}\chi(1)\chi=\sum_{X\in\mathcal X_G}\sigma_X$, the formula follows from $AHC1$ and the obvious identity
$\langle \sigma_X,\sigma_X \rangle = \sigma_X(1)$.
\end{proof}

\begin{teor}(Heilbronn-Stark Lemma)\label{t2}

For every subgroup $H$ of $G$, one has $\Theta_G|_H=\Theta_H$.
\end{teor}

\begin{proof}
By Super Frobenius Reciprocity we have
\begin{align*}
& \Theta_G|_H= \sum_{X\in \mathcal X_G} \frac{n(G,\sigma_X)}{\sigma_X(1)}\sigma_X|_H = 
\sum_{X\in \mathcal X_G} \frac{n(G,\sigma_X)}{\sigma_X(1)}  \sum_{Y\in\mathcal X_H} \langle \sigma_X|_H,\sigma_Y \rangle\sigma_Y = \\
& = \sum_{X\in \mathcal X_G} \frac{n(G,\sigma_X)}{\sigma_X(1)}\sum_{Y\in\mathcal X_H} \langle \sigma_X,\Sind_H^G\sigma_Y^G \rangle\sigma_Y =
\sum_{Y\in\mathcal X_H} \left( \sum_{X\in \mathcal X_G} \frac{n(G,\sigma_X)}{\sigma_X(1)}\langle \sigma_X,\Sind_H^G\sigma_Y^G \rangle \right) \sigma_Y
\end{align*}
On the other hand, from $AHC1$ it follows that
$$ \sum_{X\in \mathcal X_G} \frac{n(G,\sigma_X)}{\sigma_X(1)}\langle \sigma_X,\Sind_H^G\sigma_Y^G \rangle  =  n\left( G, \sum_{X\in \mathcal X_G} \frac{1}{\sigma_X(1)} \langle \sigma_X,\Sind_H^G\sigma_Y^G \rangle \sigma_X \right) 
= n ( G, \Sind_H^G\sigma_Y^G). $$
\end{proof}

\begin{teor}(Uchida-van der Waall Theorem)\label{t3}

Let $H$ be a subgroup of $G$ such that
$\Sind_H^G 1_H = 1_G + \sum_{i\in I} \Sind_{H_i}^G \sigma_i,$
where $H_i$ are subgroups of $G$, $\sigma_i$ is a supercharacters of $H_i$ with linear constituents and $I$ is a finite set of indices.
Then $$ n(H,1_H) \geq  n(G,1_G). $$
\end{teor}

\begin{proof}
From $ACH1$ and hypothesis it follows that
\begin{equation}\label{ec1}
 n(G,\Sind_H^G 1_H) = n(G,1_G) + \sum_{i\in I} n(G,\Sind_{H_i}^G \sigma_i) 
\end{equation}
From $ACH2$ and \eqref{ec1} it follows that
\begin{equation}\label{ec2}
 n(H,1_H) = n(G,1_G) + \sum_{i\in I}n(H_i,\sigma_i).
\end{equation}
The conclusion follows from \eqref{ec2} and $ACH3$.
\end{proof}

\begin{obs}\rm
If we consider the classical theory on all the subgroups of $G$, then the hypothesis of Theorem \ref{t3} is satisfied when $G$ is solvable,
according to \cite[Lemma 2.4]{murty}. 
\end{obs}

\end{document}